\documentclass{article}
\usepackage[]{amsthm,amsmath,amsfonts,amstext,amssymb}
\usepackage[dvips]{graphicx}

\newtheorem{theorem}{Theorem}
\newtheorem{lemma}{Lemma}

\begin{document}
\title{$m_3^3$-Convex Geometries are $A$-free.}
\author{J. C\'{a}ceres, Dept. of Stats. and Applied Math.\\ Univ. of Almeria, 04120, Almeria
Spain\\ e-mail: jcaceres@ual.es
\and O.R. Oellermann\thanks{Supported by an NSERC grant CANADA.},
 Dept. of Math. and Stats., Univ. of Winnipeg, \\
 515 Portage Ave, Winnipeg, R3B 2E9, Canada\\
 e-mail: o.oellermann@uwinnipeg.ca
\and M.L. Puertas, Dept. of Stats. and Applied Math.,\\ Univ. of Almeria, 04120, Almeria, Spain \\
e-mail: mpuertas@ual.es}
\maketitle

\begin{abstract}
Let $V$ be a finite set and $\cal{M}$ a collection of subsets of
$V$. Then $\cal{M}$ is an alignment of $V$ if and only if
$\cal{M}$ is closed under taking intersections and contains both
$V$ and the empty set. If $\cal{M}$ is an alignment of $V$, then
the elements of $\cal{M}$ are called convex sets and the pair $(V,
\cal{M})$ is called an aligned space. If $S \subseteq V$, then the
convex hull of $S$ is the smallest convex set that contains $S$.
Suppose $X \in \cal{M}$. Then $x \in X$ is an extreme point for
$X$ if $X \setminus \{x\} \in \cal{M}$. The collection of all
extreme points of $X$ is denoted by $ex(X)$. A convex geometry on
a finite set is an aligned space with the additional property that
every convex set is the convex hull of its extreme points. Let $G=(V,E)$
be a connected graph and $U$ a set of vertices of $G$. A subgraph $T$ of $G$ containing $U$ is a minimal $U$-tree if $T$ is a tree and if every vertex of $V(T)\setminus U$ is a cut-vertex of the subgraph induced by $V(T)$. The monophonic interval of $U$ is the collection of all vertices of $G$ that belong to some minimal $U$-tree. A set $S$ of vertices in a graph is $m_k$-convex if it contains the monophonic interval of every $k$-set of vertices is $S$.  A set of vertices $S$ of a graph is $m^3$-convex if for every pair $u,v$ of vertices in $S$, the vertices on every induced path of length at least 3 are contained in $S$. A set $S$ is $m_3^3$-convex if it is both $m_3$- and $m^3$- convex. We show that if the $m_3^3$-convex sets form a convex geometry, then $G$ is $A$-free.
\smallskip

\end{abstract}

\section{Introduction}

Let $G$ and $F$ be graphs.
Then $F$ is an {\em induced subgraph} of $G$ if $F$ is a subgraph
of $G$ and for every $u, v \in V(F)$, $uv \in E(F)$ if and only if
$uv \in E(G)$. We say a graph $G$ is $F$-{\em free} if it does not
contain $F$ as an induced subgraph. Suppose $\cal{C}$ is a
collection of graphs. Then $G$ is $\cal{C}$-free if $G$ is
$F$-free for every $F \in \cal{C}$. If $F$ is a path or cycle that is a subgraph of $G$, then $F$ has a {\em chord} if it is not
an induced subgraph of $G$, i.e., $F$ has two vertices that are
adjacent in $G$ but not in $F$. An induced cycle of length at
least $5$ is called a {\em hole}.

Let $V$ be a finite set and $\cal{M}$ a collection of subsets of
$V$. Then $\cal{M}$ is an {\em alignment} of $V$ if and only if
$\cal{M}$ is closed under taking intersections and contains both
$V$ and the empty set. If $\cal{M}$ is an alignment of $V$, then
the elements of $\cal{M}$ are called {\em convex sets} and the
pair $(V, \cal{M})$ is called an {\em aligned space}. If $S
\subseteq V$, then the {\em convex hull} of $S$ is the smallest
convex set that contains $S$. Suppose $X \in \cal{M}$. Then $x \in
X$ is an {\em extreme point} for $X$ if $X \setminus \{x\} \in
\cal{M}$. The collection of all extreme points of $X$ is denoted
by $ex(X)$. A {\em convex geometry} on a finite set $V$ is an aligned
space $(V, \cal{M})$ with the additional property that every convex set is the
convex hull of its extreme points. This property is referred to as
the {\em Minkowski-Krein-Milman }  ($MKM$) property. For a
more extensive overview of other abstract convex structures see
\cite{V}. Convexities associated with the vertex set of a graph are
discussed for example in \cite{BLS}. Their study is of interest in
Computational Geometry and has applications in Game Theory \cite{BE}.

Convexities on the vertex set of a graph are usually defined
in terms of  some type of `intervals'. Suppose $G$ is a connected
graph and $u,v$ two vertices of $G$.
Then a $u-v$ {\em geodesic} is a shortest $u-v$ path in $G$.
Such geodesics are necessarily induced paths. However, not all
induced paths are geodesics. The $g$-{\em interval} (respectively,
$m$-{\em interval}) between a pair $u, v$ of vertices in a graph
$G$ is the collection of all vertices that lie on some $u-v$
geodesic (respectively, induced $u-v$ path) in $G$ and is denoted by $I_g[u, v]$
(respectively, $I_m[u, v]$).

A subset $S$ of vertices of a graph is said to be $g$-{\em convex}
($m$-{\em convex}) if it contains the $g$-interval ($m$-interval)
between every pair of vertices in $S$. It is not difficult to see
that the collection of all $g$-convex ($m$-convex) sets is an
alignment of $V$.  A vertex $v$ is an extreme point for a $g$-convex (or
$m$-convex) set $S$ if and only if $v$ is simplicial in the
subgraph induced by $S$, i.e., every two neighbours of $v$ in $S$
are adjacent. Of course the convex hull of the extreme
points of a convex set $S$ is contained in $S$, but equality holds
only in special cases. In \cite{FJ} those graphs for which the
$g$-convex sets form a convex geometry are characterized as the
chordal $3$-fan-free graphs(see Fig.~\ref{Specialgraphs}). These
are precisely the chordal, distance-hereditary graphs (see \cite{BM, H}).
In the same paper it is shown that the chordal graphs are precisely
those graphs for which the $m$-convex sets form a convex geometry.

For what follows we use $P_k$ to denote an induced path of order
$k$. A vertex is simplicial in a set $S$ of vertices if and only
if it is not the centre vertex of an induced $P_3$ in $\langle S \rangle$.
Jamison and Olariu \cite{JO} relaxed this condition. They defined
a vertex to be {\em semisimplicial} in $S$ if and only if it is
not a centre vertex of an induced $P_4$ in $\langle S \rangle$.

\begin{figure}[htbp]
\center\includegraphics*{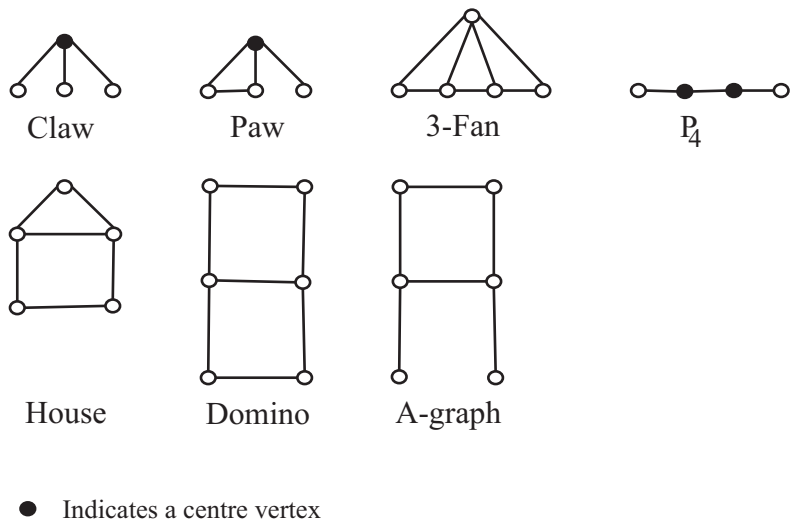} \caption{Special
Graphs }\label{Specialgraphs}
\end{figure}

Dragan, Nicolai and Brandst\"{a}dt \cite{DNB} introduced another
convexity notion that relies on induced paths. The $m^3$-{\em
interval} between a pair $u,v$ of vertices in a graph $G$, denoted
by $I_{m^3}[u,v]$, is the collection of all vertices of $G$ that
belong to an induced $u-v$ path of length at least $3$. Let $G$ be
a graph with vertex set $V$. A set $S \subseteq V$ is $m^3$-{\em
convex} if and only if for every pair $u, v$ of vertices of $S$
the vertices of the $m^3$-interval between $u$ and $v$ belong to
$S$. As in the other cases the collection of all
$m^3$-convex sets is an alignment. Note that an $m^3$-convex set
is not necessarily connected. It is shown in \cite{DNB} that the extreme points
of an $m^3$-convex set are precisely the semisimplicial vertices
of $\langle S \rangle$. Moreover, those graphs for which the
$m^3$-convex sets form a convex geometry are characterized in \cite{DNB} as
the (house, hole, domino, $A$)-free graphs (see Fig. \ref{Specialgraphs}).

More recently a graph convexity
that generalizes $g$-convexity was introduced (see \cite{O}).
The {\em Steiner interval} of a set $S$ of vertices in a connected
graph $G$, denoted by $I(S)$, is the union of all vertices of $G$
that lie on some {\em Steiner tree} for $S$, i.e., a connected
subgraph that contains $S$ and has the minimum number of edges among
all such subgraphs. Steiner intervals have been studied for example
in \cite{KKO, OP}. A set $S$ of vertices in a graph $G$ is $k$-{\em Steiner convex}
($g_k$-convex) if the Steiner interval of every collection of $k$
vertices of $S$ is contained in $S$. Thus $S$ is $g_2$-convex if
and only if it is $g$-convex. The collection of $g_k$-convex sets
forms an aligned space. We call an extreme point of a $g_k$-convex set a $k$-{\em Steiner simplicial} vertex, abbreviated $kSS$ vertex.

The extreme points of $g_3$-convex sets $S$, i.e., the $3SS$ vertices are characterized in \cite{CO} as those vertices that are {\bf not} a centre vertex of an induced claw, paw or $P_4$, in
$\langle S \rangle$ see Fig.~\ref{Specialgraphs}. Thus a $3SS$ vertex is semisimplicial. Apart from the $g_k$-convexity, for a fixed $k$, other graph convexities that (i) depend on more than one value of $k$ and (ii) combine the $g_3$ convexity and the geodesic counterpart of the $m^3$-convexity were introduced and studied in \cite{NO}. In particular characterizations of convex geometries for several of these graph convexities are given.

The notion of an induced path between a pair of vertices can be extended to three or more vertices. This gives rise to graph convexities that extend the $m$-convexity. Let $U$ be a set of at least two vertices in a connected graph $G$. A subgraph $H$ containing $U$ is a {\em minimal} $U$-{\em tree} if $H$ is a tree and if every vertex $v \in V(H) \setminus U$ is a cut-vertex of $\langle V(H) \rangle$. Thus if $U = \{u,v\}$, then a minimal $U$-tree is just an induced $u-v$ path. Moreover, every Steiner tree for a set $U$ of vertices is a minimal $U$-tree. The collection of all vertices that belong to some minimal $U$-tree is called the {\em monophonic interval of} $U$ and is denoted by $I_{m}(U)$. A set $S$ of vertices is $k$-{\em monophonic convex}, abbreviated as $m_k$-convex, if it contains the monophonic interval of every subset $U$ of $k$ vertices of $S$. Thus a set of vertices in $G$ is a monophonic convex set if and only if it is a $m_2$-convex set.
By combining the $m_3$- convexity with the $m^3$-convexity introduced in \cite{DNB}, we obtain a graph convexity that extends the graph convexity studied in \cite{NO}. More specifically we define a set $S$ of vertices in a connected graph to be $m^3_3$-{\em convex} if $S$ is both $m^3$- and $m_3$-convex.  In this paper we show  that if the $m^3_3$-convex alignment forms a convex geometry then $G$ is $A$-free.
We use the fact that these graphs are $F$-free for several other graphs $F$. In particular $G$ is easily seen to be house, hole, and domino free. Moreover the graphs of Fig. \ref{g^{(3,2)}convexgeometriesforbidden} are forbidden.  A graph $G$ is a {\em replicated twin} $C_4$ if it is isomorphic
to any one of the four graphs shown in
Fig.~\ref{g^{(3,2)}convexgeometriesforbidden}(a), where any subset
of the dashed edges may belong to $G$. The collection of the four
replicated twin $C_4$ graphs is denoted by $\cal{R}$$_{C_4}$. A graph $F$ is a {\em tailed twin} $C_4$ if it is isomorphic to one
of the two graphs shown in Fig.
\ref{g^{(3,2)}convexgeometriesforbidden}(b) where again any subset
of the dotted edges may be chosen to belong to $F$. We denote the
collection of tailed twin $C_4$'s by $\cal{T}$$_{C_4}$.

\begin{figure}[htbp]
\center\includegraphics*{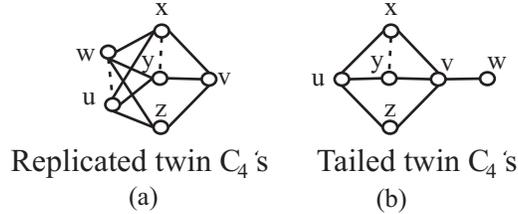}
\caption{Forbidden subgraphs for $m^3_3$-convex geometries}
\label{g^{(3,2)}convexgeometriesforbidden}
\end{figure}

\section{$m^3_3$-Convex Geometries are $A$-Free}
Recall that the graphs for which the $m^3$-convex sets form a
convex geometry are characterized in \cite{DNB} as the (house, hole, domino, $A$)-free graphs. The proof of this characterization depends on the following useful result also proven in \cite{DNB}:

\begin{theorem}\label{m^3closure}
If $G$ is a (house, hole, domino, $A$)-free graph, then every vertex of $G$ is either
semisimplicial or lies on an induced path of length at least $3$
between two semisimplicial vertices.
\end{theorem}

In \cite{DNB} several `local' convexities related to the $m^3$-convexity were studied. For a set $S$ of vertices in a graph $G$, $N[S]$ is $S \cup N(S)$ where $N(S)$ is the collection of all vertices adjacent with some vertex of $S$. A set $S$ of vertices in a graph is connected if $\langle S \rangle$ is connected. The following useful result was established in \cite{DNB}.

\begin{theorem} \label{DNBlocal}
A graph G is (house, hole, domino)-free if and only if $N[S]$ is $m^3$-convex for all connected sets $S$ of vertices of $G$.
\end{theorem}

\medskip
\begin{theorem}\label{agraph}
If $G=(V,E)$ is a graph such that $(V, \cal{M}$$_{m^3_3}(G))$ is a convex geometry, then $G$ is $A$-free.
\end{theorem}
\begin{proof} Observe first that $G$ is (house, hole, domino, $\cal{R}$$_{C_4}$, $\cal{T}$$_{C_4}$)-free.
Suppose $F$ is a house, hole, domino,
replicated twin $C_4$ or a tailed twin $C_4$. Then $F$ has at most one
$3SS$ vertex. Suppose $G$ is a graph that contains $F$ as an
induced subgraph. Then the set of extreme points of the convex
hull of $V(F)$ is contained in the collection of $3SS$ vertices of
$F$. So the convex hull of the extreme points of
the $m^3_3$-convex hull of $V(F)$ is empty or consists of a single
vertex. So in this case the $m^3_3$-convex alignment of $G$
does not form a convex geometry.

 If $S$ is a set of vertices of a graph $G$, then $I_{m^3}(S) = \cup\{I_{m^3}[x,y]|x,y \in S\}$.

 To show that $G$ contains no $A$ as an induced subgraph we prove a series of lemmas.

\begin{lemma} \label{m-interval is m^3convex}
Suppose $G=(V,E)$ is a graph for which $(V, \cal{M}$$_{m_3^3}(G))$ is a convex geometry. Then for every $a,b \in V$, $I_{m^3}(I_m[a,b]) \subseteq I_m[a,b]$.
\end{lemma}
\begin{proof}
By the above observation $G$ is (house, hole, domino, $\cal{R}$$_{C_4}$, $\cal{T}$$_{C_4}$)-free. If $ab \in E$ then $I_{m^3}(I_m[a,b]) \subseteq I_m[a,b]=\{a,b\}$. So we may assume $ab \not\in E$. If $I_{m^3}(I_m[a,b]) \not\subseteq I_m[a,b]$, there is a vertex $ w \not\in I_m[a,b]$ that lies on an induced path between two vertices of $I_m[a,b]$. Among all such induced paths of length at least $3$ containing $w$, let $Q$ be one with a minimum number of edges. Suppose $Q$ is a $u-v$ path. Clearly $\{u,v\} \neq \{a,b\}$; otherwise, $w \in I_m[a,b]$. Let $Q:(u=)v_1v_2 \ldots v_k(=v)$. (Suppose $w=v_i$.) Then $w$ is not adjacent with two non-adjacent vertices of any induced $a-b$ path; otherwise, $w$ lies on an induced $a-b$ path.

\noindent{\bf Case 1}  Suppose $u$ and $v$ lie on a common induced $a-b$ path $P$. We may assume $u$ precedes $v$ on such a path. Moreover, we may assume that all internal vertices of $Q$ are not on $P$. For if $v_j \in V(P)$, $1 < j <k$, then either $Q[v_1,v_j]$ or $Q[v_j, v_k]$ contains $w$, say the former. Since $Q$ is an induced path, so is $Q[v_1,v_j]$. Hence $v_1v_j \not\in E$. Thus $Q[v_1,v_j]$ must have length at least $3$; otherwise $w$ is adjacent with a pair of nonadjacent vertices of $P$, implying that $G$ contains an induced $a-b$ path passing through $w$, contrary to assumption. But then we have a contradiction to our choice of $Q$.

Let $S_1 = P[u,v] \setminus\{u,v\}$ and $S_2 = Q[u,v] \setminus \{u,v\}$. Then $\langle S_i \rangle $ is connected for $i=1,2$. By Theorem \ref{DNBlocal}, $N[S_i]$ is $m^3$-convex. Since $u$ and $v$ both belong to $N[S_i]$, every vertex of $Q$ must be adjacent with an internal vertex of $P[u,v]$. This is true in particular for  $w$. Since $P[a,u]$ followed by Q and then $P[v,b]$ is an $a-b$ path that contains $w$ it cannot be induced. Some vertex of $P[a,u] \setminus\{u\}$ or a vertex of $P[v,b] \setminus\{v\}$ must be adjacent with an internal vertex of $Q$; say the former occurs. Let $x$ be the first vertex of $P[a,u]$ that is adjacent with an internal vertex $y$ of $Q$. Let $r$ be the first vertex on $Q[y,v]$ that is adjacent with a vertex of $P[v,b]$ (possibly $r$ is $v_{k-1}$). Let $s$ be the last vertex of $P[v,b]$ adjacent with $r$. Then the path $H: P[a,x]xyQ[y,r]rsP[s,b]$ is an induced $a-b$ path and thus does not contain $w$. So $w$ is an internal vertex of $Q[u,y]$ or of $Q[r,v]$; suppose the former. Since $H$ is connected, $N[V(H)]$ is $m^3$-convex by Theorem \ref{DNBlocal}. Since $a,b \in N[V(H)]$ and as $P$ has length at least $3$, $N[V(H)]$ must contain every vertex of $P$. Thus $I_{m^3}[u,v] \subseteq N[V(H)]$. Hence $w$ is adjacent with a vertex of $H$. Since $w$ is adjacent with an internal vertex of $P[u,v]$, $w$ is not adjacent with any vertex of $P[a,x]$ nor $P[s,b]$. Since $Q$ is an induced path, the only vertex of $H$ to which $w$ can be adjacent is $y$. So $y$ follows $w$ on $Q$. Since $u$ and $y$ belong to $I_m[a,b]$ and as $Q[u,y]$ is an induced path containing $w$, it follows that $w$ must be adjacent with $u$; otherwise, we have a contradiction to our choice of $Q$. Let $x'$ be the last vertex on $P[x,u]$ to which $y$ is adjacent. Then $x'u \in E$; otherwise $P[x',u]uwyx'$ is an induced cycle of length at least $5$. Let $z$ be the first internal vertex of $P[u,v]$ to which $w$ is adjacent. (By an earlier observation $z$ exists.) Then $uz \in E$; otherwise, $w$ lies on an induced $a-b$ path. Also $yz \in E$; otherwise, $\langle \{x',u,w,y,z\} \rangle$ is a house. If $r \ne y$, let $y'$ be the neighbour of $y$ on $Q[y,r]$. Then $u,y' \in I_m[a,b]$ and $Q[u, y']$ is an induced path between two vertices of $I_m[a,b]$ having length $3$ and containing $w$, contrary to our choice of $Q$. So $r=y$. So $P[x', s]syx'$ is a cycle of length at least $5$. Since $yu \not\in E$, $x'uzyx'$ is an induced $4$-cycle. Let $z'$ be the first vertex after z on $P[z,s]$ to which $y$ is adjacent (perhaps $z' = s$). Then $P[z,z']z'yz$ is an induced cycle and hence has length $3$ or $4$. This cycle together with the $4$-cycles $x'yzux'$ produces either a house or a domino both of which are forbidden. So we may assume that $Q$ is an induced $u-v$ path between vertices $u$ and $v$ of $I_m[a,b]$ that do not belong to the same induced $a-b$ path. Indeed we may assume if $u$ and $v$ are any non-adjacent vertices that lie on the same induced $a-b$ path, then  $I_{m^3}[u,v] \subseteq I_m[a,b]$

\noindent{\bf Case 2}  Suppose $u$ and $v$ lie on two internally disjoint $a-b$ paths $P_u$ and $P_v$, respectively. We may assume $\{u,v \} \cap \{a,b \} = \emptyset$; otherwise, we are in Case 1.

\noindent{\em We show first that no internal vertex of $Q$ belongs to $P_u$ or $P_v$.} Suppose some internal vertex of $Q[u,w]$ or $Q[w,v]$, say $Q[u,w]$ belongs to $P_u$ or $P_v$. However, no internal vertex of $Q[u,w]$ belongs to $P_v$; otherwise, either the situation arises that was considered in Case 1 or there is an induced $a-b$ path containing $w$. So we may assume that an internal vertex of $Q[u,w]$ lies on $P_u$. Let $u'$ be the last such vertex. Then $Q[u',v]$ contains $w$ and is an induced path between two vertices of $I_m[a,b]$ that is shorter than $Q$. So $Q[u', v]$ has length $2$; otherwise we have a contradiction to our choice of $Q$. So $Q[u', v]$ must be the path $u'wv$. Since $Q$ has length at least $3$ and by our choice of $Q$ one of the neighbours of $u'$ on $P_u$ must be $u$. So one of the configurations shown in Fig. \ref{two configurations}  must occur where solid lines are edges and dashed lines represent subpaths of $P_u$ and $P_v$. We may assume that the configuration in (a) occurs. The argument for the configuration in (b) is similar.

\begin{figure}[htbp]
\center\scalebox{.60}{\includegraphics*{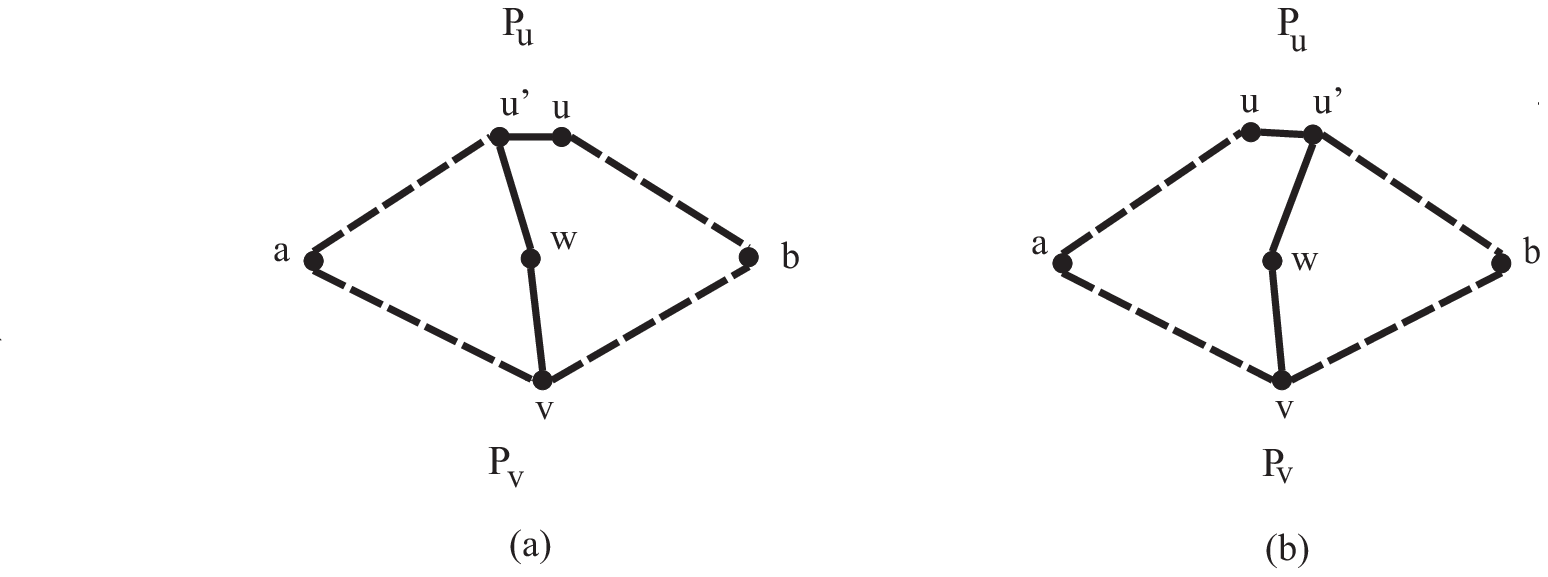}}
\caption{Two configurations in Case 2}
\label{two configurations}
\end{figure}

Since $Q$ is induced, $v$ is not adjacent to $u'$ or $u$ and $w$ is not adjacent with $u$. Let $v_L$ and $v_R$ be the neighbours of $v$ on $P_v[a,v]$ and $P_v[v,b]$, respectively. If $u'$ is adjacent with a vertex $r$ of $P_v[v_R, b] - v_R$ then $ru'wv$ is an induced path of length $3$ containing $w$ and whose end vertices lie on the same induced $a-b$ path. By Case 1, this situation cannot occur.  So the only vertex of $P_v[v_R, b]$ to which $u'$ can be adjacent is $v_R$. Similarly, the only vertex of $P_v[a, v_L]$ to which $u'$ can be adjacent is $v_L$. Using a similar argument and the fact that $vu \not\in E$, we see that $v$ is not adjacent with any vertex of $P_u[u,b]$. Moreover, $w$ is not adjacent with any vertex of $P_u[u,b]$; otherwise, $w$ lies on an induced $a-b$ path. The path obtained by taking $P_v[a,v]$ followed by $vwu'$ and then $P_u[u',b]$ is an $a-b$ path that contains $w$. Hence this path is not induced. {\em Suppose first that} $wv_L \notin E$. So some vertex of $P_v[a,v]$ is adjacent with some vertex of $P_u[u', b]$. Since $v$ is not adjacent with any vertex of $P_u[u', b]$, some vertex of $P_v[a,v_L]$ is adjacent with some vertex of $P_u[u', b]$. Let $z$ be a vertex closest to $v$ on $P_v[a,v]$ that is adjacent with a vertex of $P_u[u',b]$ and let $y$ be such a neighbour of $z$ closest to $u'$ on $P_u[u', b]$. Observe that $y=u'$ and $z=v_L$; otherwise, the cycle $P_v[z, v]vwu'P_u[u'y]yz$ is an induced cycle of length at least $5$. Let $x$ be the vertex closest to $u'$ on $P_u[u',b]$ that is adjacent with a vertex of $P_v[v, b]$ (possibly $x=b$). Let $x'$ be the neighbour of $x$ on $P_v[v,b]$ closest to $v$. By the above observation $x' \ne v$. The cycle $P_u[u',x]xx'P_v[x',v]vwu'$ is induced and has length at least $5$ unless $x=u'$ and  $x'=v_R$. So $u'$ is adjacent with both $v_L$ and $v_R$. Observe that $u$ is either adjacent with both $v_L$ and $v_R$ or neither of these two vertices; otherwise, $\langle\{v_L, v, v_R, u', u \} \rangle$ is a house. We show next that no vertex of $P_u[u,b]$ is adjacent with $v_L$. Suppose $r$ is a vertex on $P_u[u,b]$ closest to $u$ that is adjacent with $v_L$ First observe that $r\ne u$ for if $uv_L \in E$, then $\langle\{u,u',w,v,v_L\}\rangle$ is a house. So $r$ must be the neighbour of $u$ on $P_u[u,b]$; otherwise, $G$ has a hole. However then $\langle\{u',u,r,v_L,v,w\} \rangle$ is a domino. So $v_L$ is not adjacent with any vertex of $P_u[u,b]$. Let $C: v_Ru'P_u[u',b]P_v[b,v_R]$. Then $C$ is a cycle of length at least $5$ and hence has chords. Now $u'$ is not adjacent with any vertex of $P_v[v_R, b]$ other than $v_R$; otherwise, $w$ lies on an induced path of length $3$ between two vertices of $I_m[a,b]$ that belong to the same induced $a-b$ path, a case already dealt with. Since $uv_L \not\in E$, $uv_R \not\in E$. Suppose $u$ is adjacent with an internal vertex of $P_v[v_R,b]$. Let $s$ be such a vertex closest to $v_R$.  So $s \ne v_R$. Since $G$ contains no holes, $s$ is adjacent with $v_R$. But then $\langle \{ u',u,s,v_R,v,v_L \} \rangle$ is a domino. So the neighbour $r$ of $u$ on $P_u[u,b]$ is incident with a chord of $C$. Since $G$ has no holes $rv_R \in E$. But then $\langle\{u',u,r,v_R,v,v_L\} \rangle$ is a domino. {\em Suppose now that} $wv_L \in E$. Then $wv_R \not\in E$. Let $C':P_u[u', b]P_v[b,v]vwu'$. Then $C'$ is a cycle of length at least $5$ and hence has no chords. Since neither $w$ nor $v$ are incident with chords of $C'$, $u'v_R \in E$. If $uv_R \in E$ $\langle \{u,u',v_R,v,w \} \rangle$ is a house. Note that $u'$ is not adjacent with an internal vertex vertex of $P_v[v_R,b]$; otherwise, if $t$ is such a neighbour of $u'$, then $tu'wv$ is an induced path of length $3$ between two vertices of $I_m[a,b]$ that lie on the same induced $a-b$ path, a case already considered. Let $r$ be the neighbour of $u$ on $P_u[u,b]$ and $s$ the neighbour of $v_R$ on $P_v[v_R, b]$. Then either $v_Rr$ or $us$ is an edge; otherwise, $G$ has a hole. But then $\langle \{u',u,r, v_R, v, w \} \rangle$ or $\langle\{ u',u,s,v_R,v,w \} \rangle$ is a domino. So no internal vertex of $Q$ belongs to $P_u$ or to $P_v$.

Let $Q:(u=)v_1v_2 \ldots v_k(=v)$. Let $u_L$ and $u_R$ be the neighbours of $u$ on $P_u[a,u]$ and $P_u[u,b]$, respectively and $v_L$ and $v_R$ the neighbours of $v$ on $P_v[a,v]$ and $P_v[v, b]$, respectively. Let $S_1 = V(P_u[u_R, b]) \cup V(P_v[b, v_R])$ and $S_2 =V(P_u[a, u_L]) \cup V(P_v[a, v_L])$. Since $\langle S_i \rangle$ is connected for $i=1,2$, it follows from Theorem \ref{DNBlocal} that $N[S_i]$ is $m^3$-convex. Since $u,v \in N[S_i]$ for $i=1,2$, every vertex of $Q$ is adjacent with a vertex of $S_i$ for $i=1,2$. In particular $w$ is adjacent with a vertex of $S_i$ for $i=1,2$. However, $w$ is not adjacent with a pair of nonadjacent vertices of $P_u$ nor a pair of nonadjacent vertices of $P_v$. So without loss of generality we may assume that $w$ is adjacent with a vertex of $P_v[v_R, b]$ and a vertex of $P_u[a, u_L]$. Also $w$ is not adjacent with either $a$ or $b$; otherwise, $w$ lies on an induced $a-b$ path.

If $v_2$ is adjacent with two non-adjacent vertices of $P_u$ (or if $v_{k-1}$ is adjacent with two nonadjacent vertices of $P_v$), then $v_2 \ne w$ ( and $v_{k-1} \ne w$, respectively) and $Q[v_2, v]$ (or $Q[u, v_{k-1}]$, respectively) is an induced path between two vertices of $I_m[a, b]$ that is shorter than $Q$ and contains $w$. By our choice of $Q$ this can only happen if $Q$ has length $3$.

We consider two subcases that depend on the length of $Q$.\\
\noindent{\bf Subcase 2.1} Suppose $Q$ has length $3$. \\
Then $v_2$ or $v_3$ is $w$, say $v_3=w$. The case where $v_2 = w$ can be argued similarly. From the above, we may assume that $w$ is adjacent with an internal vertex of $P_v[v,b]$ and an internal vertex of $P_u[a,u]$. The only vertex of $P_v[v_R,b]$ that can be adjacent with $w$ is $v_R$; otherwise, $w$ lies on an induced $a-b$ path. So $wv_R \in E$. Now it follows that $w$ is not adjacent with a vertex of $P_v[a,v_L]$. Thus $\langle\{v_2,w,v,v_L,v_R \} \rangle$ is a house unless $v_2v_R \in E$. If $v_2v_R \in E(G)$, then $uv_L, ~ uv_R \not\in E$; otherwise, $\langle\{u, v_2, v_3, v, v_L \} \rangle$ or $\langle \{ u, v_2, v_3, v, v_R \} \rangle$ is a house. So $\langle \{ u, v_2, v_3, v_L, v, v_R \} \rangle$ is a tailed twin $C_4$ which is forbidden.  So this subcase cannot occur.

\noindent{\bf Subcase 2.2} Suppose $Q$ has length at least $4$. \\
By an earlier observation, $v_2$ is not adjacent with a pair of non-adjacent vertices of $P_u$ and $v_{k-1}$ is not adjacent with a pair of non-adjacent vertices of $P_v$. By assumption, $w$ is adjacent with an internal vertex of $P_u[a,u]$ and an internal vertex of $P_v[v,b]$. Suppose $w=v_j$. So $w$ is not adjacent with a vertex of $P_u[u_R,b]$ nor a vertex of $P_v[a,v_L]$.\\

\noindent{\bf Fact 1}  {\em No vertex of $Q[v_1, v_{j-1}]$ is adjacent with a vertex of $P_v[a,v_L]$ and no vertex of $Q[v_{j+1}, v_k]$ is adjacent with a vertex of $P_u[u_R, b]$.}\\
\noindent{\em Proof of Fact 1.}  Suppose some vertex of $Q[v_1, v_{j-1}]$ is adjacent with a vertex of $P_v[a, v_L]$. Let $i$ be the largest integer less than $j$ such that $v_i$ is adjacent with a vertex of $P_v[a, v_L]$. Let $z$ be a neighbour of $v_i$ on $P_v[a,v_L]$ closest to $v$ on this path. Then $C_1: Q[v_i,v]P_v[v,z]v_i$ is a cycle of length at least $4$. If $i \le j-2$, then $C_1$ has length at least $5$ and three consecutive vertices of $C_1$ are not incident with a chord of the cycle. This implies that $G$ has a hole; which is forbidden. So $i=j-1$. Clearly $j \le k-1$. Let $C_2: P_v[z,v]Q[v,v_{j+1}]z$. Then $C_2$ is a cycle of length at least $3$. Thus $\langle V(C_2) \rangle$ contains an induced cycle $C'$ of length at least $3$ that contains the edge $zv_{j+1}$. Since $G$ contains no holes, $C'$ has length $3$ or $4$. Since neither $v_j$ nor $v_{j-1}$ is adjacent with a vertex of $P_v[z,v]-z$ nor a vertex of $Q[v_{j+2}, v]$ and as $v_jz \notin E$, it is not difficult to see that the vertices of $C_2$ and $C'$ induce a house or a domino. So no vertex of $Q[v_1,v_{j-1}]$ is adjacent with a vertex of $P_v[a,v_L]$. By an identical argument we can show that no vertex of $Q[v_{j+1}, v_k]$ is adjacent with a vertex of $P_u[u_R, b]$. $\Box$\\

\noindent{\bf Fact 2} {\em No vertex of $P_v[a, v_L]$ is adjacent with any vertex of $P_u[u_R, b]$.}\\
\noindent{\em Proof of Fact 2.} Let $z$ be the first vertex of $P_v[a, v_L]$ that is adjacent with some vertex of $P_u[u_R, b]$. Let $y$ be a neighbour of $z$ on $P_u[u_R, b]$ that is closest to $b$. Then the path $P:P_v[a,z]zyP_u[y,b]$ is an induced $a-b$ path. So $N[V(P)]$ is $m^3$-convex and hence contains all induced $a-b$ paths of length at least $3$. Since $\{a,b\} \cap \{u,v\} =\emptyset$, and since both $P_u[a, u]$ and $P_v[v,b]$ contain an internal vertex adjacent with $w$, both $P_u$ and $P_v$ have length at least $3$. So $N[V(P)]$ contains all the vertices of $P_u$ and $P_v$ and hence $u$ and $v$. So $N[V(P)]$ also contains $Q$. Thus every vertex of $Q$ is adjacent with  a vertex of $P_v[a,z]$ or with a vertex of $P_u[y,b]$. But by assumption $w$ is adjacent with an internal vertex of both $P_u[a, u]$ and $P_v[v,b]$. So $w$ is adjacent with a pair of non-adjacent vertices of $P_v$ or a pair of non-adjacent vertices of  $P_u$, neither of which is possible. $\Box$ \\

From Facts $1$ and $2$, it follows that no vertex of the path $P_v[a,v] Q[v, v_{j-1}]$ is adjacent with a vertex of the path $Q[v_{j+1}, u]P_u[u,b]$. Hence the subgraph induced by the path $P_v[a,v]Q[v,u]P_u[u,b]$ is an induced $a-b$ path that contains $w$; contrary to the assumption that $w \not\in I_m[a, b]$. This completes the proof of Case 2.\\

\noindent{\bf Case 3}  Suppose that $u$ belongs to an induced $a-b$ path $P_u$ and $v$ to an induced $a-b$ path $P_v$ where $P_u$ and $P_v$ intersect at vertices other than $a$ and $b$. We may assume that $u$ and $v$ do not both belong to $P_u$ nor both to $P_v$; otherwise, Case 1 occurs. Let $a'$ be the last vertex prior to $u$ on $P_u[a, u]$ that is also a vertex of $P_v$ (perhaps $a'=a$). Let $b'$ be the first vertex after $u$ on $P_u[u,b]$ that belongs to $P_v$. So $a'b' \not\in E$. Let $a''$ be the last vertex prior to $v$ on $P_v[a,v]$ that also belongs to $P_u$ and $b''$ the first vertex after $v$ on $P_v[v, b]$ that also belongs to $P_u$. So $a''b'' \not\in E$.\\

\noindent{\bf Subcase 3.1}  Suppose $P_u[a'',b'']$ contains both $a'$ and $b'$. (Note $b''$ may precede $a''$ on $P_u[a'', b'']$.) In this case we can apply the argument used in Case 2 with $a$ and $b$ replaced by $a''$ and $b''$ and $P_u$ and $P_v$ replaced by $P_u[a'',b'']$ and $P_v[a'', b'']$. Hence this subcase cannot occur. \\

\noindent{\bf Subcase 3.2} Suppose $P_u[a'', b'']$ does not contain both $a'$ and $b'$. Then $a''$ and $b''$ either lie on $P_u[a,a']$ or on $P_u[b',b]$. We will assume the former case occurs. The arguments for the latter case are similar. We may assume $a''$ precedes $b''$ on $P_u[a,a']$. The case where $b''$ precedes $a''$ on $P_u[a,a']$ is similar. First suppose that $P_v[a'',b'']$ has length $2$. Then $v$ is the only interior vertex of $P_v[a'', b'']$ and $v$ is adjacent with two nonadjacent vertices of $P_u$. Let $x$ be the first vertex on $P_u$ that is adjacent with $v$, and $y$ the last vertex of $P_u$ adjacent with $v$. Since $uv \not\in E$, $y \ne u$. If $y$ precedes $u$ on $P_u$, then the path obtained by taking $P_u[a,x]$ followed by $xvy$ and then $P_u[y,b]$ is an induced $a-b$ path that contains both $u$ and $v$. Thus we can apply the argument used in Case 1 to this path to obtain a contradiction. If $y$ follows $u$ on $P_u$, then we can use the path $P_u[x,y]$ and the path $xvy$ and apply the argument used in Case 2 with $x$ and $y$ instead of $a$ and $b$, respectively.

We now assume that $P_v[a'',b'']$ has length at least $3$. Since $H =P_u[a'',b'']\setminus\{a'',b''\}$ is connected it follows, from Theorem \ref{DNBlocal}, that $N[V(H)]$ is $m^3$-convex. Since $N[V(H)]$ contains both $a''$ and $b''$ it must contain every internal vertex of $P_v[a'', b'']$. So each internal vertex of $P_v[a'', b'']$ is adjacent with an internal vertex of $P_u[a'', b'']$. If no internal vertex of $P_v[a'', b'']$ is adjacent with a vertex of $P_u[a,a''] \setminus\{a''\}$ or $P_u[b'', b] \setminus\{b''\}$, then we can replace $P_u[a'',b'']$ in $P_u$ with $P_v[a'', b'']$ to obtain an induced $a-b$ path that contains both $u$ and $v$. By applying the argument used in Case 1 to this path we obtain a contradiction. Let $b''_L$ and $b''_R$ be the neighbours of $b''$ that precede and succeed $b''$ on $P_u$. Let $x$ be the neighbour of $b''$ on $P_v[a'',b'']$.

Suppose first that some internal vertex  $t$ of $P_v[a'', b'']$ is adjacent with some vertex  $y$ of $P_u[b''_R, b]$.  If $t \ne x$, then $t$ is also adjacent with some internal vertex $z$ of $P_u[a'', b'']$. So $t \ne v$; otherwise, $v$ is adjacent with two nonadjacent vertices of $P_u$ which leads to a situation where the arguments of either Case 1 or Case 2 apply. If $P_u[z,y]$ has length at least $3$, then it follows, from Theorem \ref{DNBlocal}, that $t$ is adjacent with every vertex of $P_u[z,y]$ including $b''$; this is not possible as $t$ and $b''$ are nonadjacent vertices on the induced path $P_u[a'', b'']$. So $z = b''_L$ and $y = b''_R$ and $b''_L$ is the only vertex of $P_u[a'', b'']$ to which $t$ is adjacent. Suppose $P_v[t, b'']$ contains $v$. If $P_v[t, b'']$ contains at least four vertices, then the subgraph induced by $b''_R$ and the vertices of $P_v[t, b'']$ must contain a hole, house or domino. (We use the fact that $v$ cannot be adjacent to nonadjacent vertices of $P_u$; otherwise, one can again argue that Case 1 or Case 2 occurs.) Suppose now that $P_v[t, b''] = tvb''$. Let $d$ be the neighbour of $b_R''$ on $P_u[b''_R, b]$. Then $\langle\{ t,v,b'',b''_L, b''_R, d \} \rangle$ is a tailed twin $C_4$ since $v$ is nonadjacent with $b''_R$ and $d$.

 Suppose thus that $v$ does not belong to $P_v[t, b'']$. Then we may assume that $t$ is the first internal vertex on $P_v[a'', x]$ that is adjacent with $b''_R$. Let $s$ be the neighbour of $t$ on $P_v[a'', t]$. By the above we know that $tb''_L \in E$. If $sb''_L \in E$, then $\langle \{ s,t, b''_L, b'', b''_R\} \rangle$ is a house which is forbidden. So assume $sb''_L \not\in E$. Let $c$ be the neighbour of $b''_L$ on $P_u[a'', b''_L]$. Since $tc \not\in E$ and $G$ has no holes, $sc \in E$. But then $\langle \{s,c,t, b''_L, b'', b''_R \} \rangle$ is a domino, which is forbidden. So $x$ is the only internal vertex of $P_v[a'', b'']$ that is adjacent with vertices of $P_u[b''_R, b]$. Let $y$ be the neighbour of $a''$ on $P_v[a'', b'']$ and let $a''_L$ and $a''_R$ be the neighbours of $a''$ on $P_u[a, a'']$ and $P_u[a'',b'']$, respectively. One can argue as in the previous situation that the only internal vertex of $P_v[a'', b'']$ that is possibly adjacent with a vertex of $P_u[a, a'']$ is $y$.

  Now let $y'$ be the first vertex on $P_u[a, a'']$ that is adjacent with $y$ (possibly $y'=a''$) and let $x'$ be the last vertex on $P_u[b'', b]$ to which $x$ is adjacent (possibly $x'=b''$). If $x'$ belongs to $P_u[b'', u]$, then the path obtained by taking $P_u[a, y']$ followed by $y'yP_v[y, x]$ and then $xx'P_u[x',b]$ is an induced $a-b$ path containing both $u$ and $v$. By Case 1 this produces a contradiction. Suppose thus that $x'$ belongs to $P_u[u, b] -u$. Then $P_u[y', x']$ and $y'yP_v[y,x]xx'$ are two internally disjoint $y'-x'$ paths containing $u$ and $v$, respectively. By applying the arguments of Case 2 to these two paths we again obtain a contradiction. Hence Case 3 cannot occur either.

\end{proof}

\begin{lemma}\label{m-interval is m_3convex}
Suppose $G=(V, E)$ is a graph for which $(V, \cal{M}$$_{m_3^3}(G))$ is a convex geometry. Then for all $a, b \in V$, $I_{m_3}(I_m[a,b]) \subseteq I_m[a, b]$.
\end{lemma}
\begin{proof}
By the above $G$ is (house, hole, domino, $\cal{R}$$_{C_4}$, $\cal{T}$$_{C_4}$)-free. If $ab\in E$, then $I_m[a,b] = \{a, b\} = I_{m^3}(\{a,b\})= I_{m^3}(I_m[a,b])$. Suppose $ab \not\in E$. So, by Lemma \ref{m-interval is m^3convex}, $I_{m^3}(I_m[a,b]) \subseteq I_m[a,b]$ (in fact equality holds). If $I_{m_3}(I_m[a,b]) \not\subseteq I_m[a, b]$, then there is a set $W =\{w_1, w_2, w_3\} \subseteq I_m[a,b]$ such that $I_{m_3}(W) \not\subseteq I_m[a,b]$. So there is an minimal $W$-tree $T$ that contains a vertex $x \not\in I_m[a, b]$. Let $H = \langle V(T) \rangle$. Then $x$ is a cut-vertex of $H$. Thus one of the vertices of $W$, say $w_3$ does not belong to the component of $H-x$ that contains $w_1$ nor the component containing $w_2$. So $x$ lies on an induced $w_3-w_i$ path for $i=1,2$. Since, by Lemma \ref{m-interval is m^3convex}, $I_m[a,b]$ is $m^3$-convex it must be the case that $x$ is adjacent with $w_1, w_2$ and $w_3$; otherwise, $x \in I_m[a,b]$. So $x$ is on an induced path between every pair of nonadjacent vertices of $W$.\\
\noindent{\bf Case 1} Suppose two nonadjacent vertices of $W$ lie on the same induced $a-b$ path $P$. Then $x$ is adjacent with a pair of nonadjacent vertices of an induced $a-b$ path. Hence $x$ lies on an induced $a-b$ path; contrary to assumption. So $w_1, w_2$ and $w_3$ cannot lie on the same induced $a-b$ path.\\

\noindent{\bf Case 2} Suppose that two adjacent vertices of $W$, say $w_1$ and $w_2$, lie on an induced $a-b$ path $P$. By Case 1, $w_3$ does not lie on the same induced $a-b$ path as $w_1$ and $w_2$. Let $Q$ be an induced $a-b$ path containing $w_3$. Let $s_3$ and $t_3$ be the neighbours of $w_3$ on $Q[a,w_3]$ and $Q[w_3, b]$, respectively. (Note that $w_3 \ne a$ or $b$; otherwise, the vertices of $W$ lie on the same induced $a-b$ path. So $s_3$ and $t_3$ are well-defined.) Since $w_1w_2 \in E$, $w_1w_3, w_2w_3 \not\in E$. Hence $\{s_3, t_3 \} \cap\{w_1,w_2\} = \emptyset$. Since $x$ cannot be adjacent with two nonadjacent vertices of $Q$, $x$ cannot be adjacent with both $s_3$ and $t_3$. We may assumer $xt_3 \not\in E$. The path $R:w_2xw_3t_3$ is a path of length $3$ between two vertices of $I_m[a,b]$. By Lemma \ref{m-interval is m^3convex}, $I_m[a,b]$ is $m^3$-convex. If $R$ is induced this would imply that $x \in I_m[a,b]$, contrary to assumption.  Hence $w_2t_3 \in E$. Now $\langle \{w_1, w_2, x, w_3, t_3\} \rangle$ is a house unless $w_1t_3 \in E$.

If $xs_3 \not\in E$, then we can argue as for $t_3$ that $s_3w_1, s_3w_2 \in E$. But then $\langle \{ w_1,w_2,w_3,x, s_3, t_3 \} \rangle$ is a replicated twin $C_4$ which is forbidden.

Suppose now that $xs_3 \in E$. Then $\langle \{s_3,w_3,t_3, w_2, x\} \rangle$ is a house unless $s_3w_2 \in E$. If $w_1s_3 \not\in E$, the path $R:s_3xw_1t_3$ is an induced path, of length $3$, between two vertices in $I_m[a,b]$ that contains $x$. Since $I_m[a,b]$ is $m^3$-convex and $R$ contains $x$ this contradicts our assumption about $x$. So $w_1s_3 \in E$. However, then $\langle\{w_1, w_2, w_3, x, s_3, t_3 \} \rangle$ is again a replicated twin $C_4$ which is forbidden. So this case cannot occur.\\

\noindent{\bf Case 3} Suppose that no two vertices of $W$ lie on the same induced $a-b$ path in $G$. (We may also assume that $w_1w_3, w_2w_3 \not\in E$.) Let $P_i$ be an induced $a-b$ path containing $w_i$ for $i= 1,2,3$. From the case we are in $w_i$ is not equal to either $a$ or $b$ for $i=1,2,3$. For $i=1,2,3$, let $s_i$ and $t_i$ be the neighbours of $w_i$ on $P_i[a, w_i]$ and $P_i[w_i,b]$, respectively.\\

\noindent{\bf Subcase 3.1} $\{s_1,t_1\} = \{s_2, t_2\} = \{s_3, t_3\}$. Since $s_1$ and $t_1$ are non-adjacent vertices of $P_1$, $x$ is adjacent with at most one of $s_1$ or $t_1$. Hence $\langle \{ w_1, w_2, w_3, s_1, t_1, x \} \rangle$ is a replicated twin $C_4$ which is forbidden. So $\{s_3, t_3\}$ is either not equal to $\{s_1, t_1\}$ or $\{s_2, t_2\}$; suppose the former.\\

\noindent{\bf Subcase 3.2} $\{s_1, t_1\} \cap \{s_3, t_3 \} = \emptyset$. Since $s_i$ and $t_i$ are non-adjacent vertices of $P_i$, $x$ cannot be adjacent with both $s_i$ and $t_i$ for $i = 1, 2, 3$. So we may assume $xt_1 \not\in E$. Suppose first that $xt_3 \not\in E$. Since $t_1w_1xw_3$ is a path of length $3$ between two vertices of $I_m[a, b]$ that contains $x$, it follows from Lemma \ref{m-interval is m^3convex} that this is not an induced path. Hence $w_3t_1 \in E$. Similarly by considering the path $w_1xw_3t_3$ and using the same argument it follows that $w_1t_3 \in E$. Similarly by considering the paths $w_2xw_3t_1$ and $w_2xw_3t_3$, it follows that $w_2t_1$ and $w_2t_3 \in E$. But now $\langle\{w_1, w_2, w_3, t_1, t_2, x \} \rangle$ is a replicated twin $C_4$ which is forbidden. So this case cannot occur.\\

\noindent{\bf Subcase 3.3} $|\{s_1, t_1\} \cap \{s_3,t_3\}| = 1$. We may assume $s_1 \in \{s_3, t_3\}$. The case where $t_1 \in \{s_3, t_3 \}$ can be argued similarly. Suppose first that $s_1 = s_3$. If $s_1x \in E$, then $xt_1, xt_3 \not\in E$. But then we can argue similarly as in Subcase 3.2 that $\langle \{ w_1, w_2, w_3, t_1, t_3, x \} \rangle$ is a replicated twin $C_4$. Hence $s_1x \not\in E$.  Suppose at least one of $xt_1$ or $xt_3$ is in $E$, say $xt_1 \in E$. Then $\langle \{s_1,w_1, t_1, x, w_3\} \rangle$ is a house unless $t_1w_3 \in E$. By considering the path $w_2xw_3s_1$ we can argue as before that $w_2s_1 \in E$. By now considering the path $t_1xw_2s_1$ it follows that $t_1w_2 \in E$. Thus $\langle \{w_1, w_2, w_3, s_1, t_1, x\} \rangle$ is a replicated twin $C_4$ which is forbidden. If neither $xt_1$ nor $xt_3$ are in $E$, then one can argue in a similar manner that $\langle \{ w_1, w_2, w_3, s_1, x, t_3 \} \rangle$ is a replicated twin $C_4$. If $s_1 = t_3$ we can argue similarly that $G$ contains a replicated twin $C_4$ which is forbidden. Hence this case cannot occur either. This completes the proof of the lemma.
\end{proof}

\begin{lemma} \label{monophonic interval of leaves of A graph}
If $G=(V,E)$ is a (house, hole, domino, $\cal{T}$$_{C_4}$)-free graph that contains an induced $A$-graph as labeled in Fig.4, then $u_2 \not\in I_m[a,b]$.
\end{lemma}
\begin{proof}

\begin{figure}[htbp]
\label{Agraph}
\center\scalebox{.60}{\includegraphics*{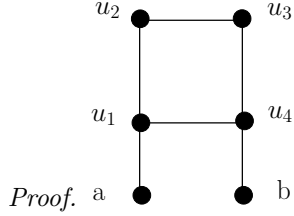}}
\caption{A labeled $A$-graph}
\end{figure}
Suppose, to the contrary, that $u_2 \in I_m[a,b]$ and let $P$ be an induced $a-b$ path containing $u_2$.\\

\noindent{\bf Case 1} $u_1 \not\in V(P[a, u_2])$. Suppose $P[a, u_2] : aw_1w_2 \ldots w_ku_2$. If $k=1$, then $\langle \{ a,w_1,u_2,u_1,u_4, u_3 \} \rangle $ is a domino unless at least one of $w_1u_4, w_1u_3,w_1u_1 \in E$.
If $w_1u_3 \not\in E$, then $w_1u_1$ or $w_1u_4 \in E$. Suppose $w_1u_1 \in E$. Then $\langle \{ w_1, u_2, u_3,$ $ u_1, u_4 \} \rangle$ is a house unless $w_1u_4 \in E$. So in either case $w_1u_4 \in E$. But then $\langle \{u_2,$ $w_1,u_1, u_3, u_4, b\} \rangle$ is a tailed twin $C_4$ which is forbidden. So $w_1u_3 \in E$. Since $\langle \{w_1,a,u_1, u_4, u_3 \} \rangle$ is not a hole, either $w_1u_1$ or $w_1u_4$ is in $E$. If $w_1u_4 \not\in E$, then $\langle \{w_1,a,u_1, u_4, u_3 \} \rangle$ is a house which is forbidden. Hence $w_1u_4 \in E$. So if $P[a, u_2]$ has length $2$, then its interior vertex is adjacent with both $u_3$ and $u_4$.

Suppose now that $k \ge 2$. By Theorem \ref{DNBlocal}, $N[u_1]$ is $m^3$-convex. Since $N[u_1]$ contains both $a$ and $u_2$, every vertex of $P[a,u_2]$ is adjacent with $u_1$. However, then $\langle \{w_k, u_1, u_2, u_3, u_4 \} \rangle$ is a house unless $w_ku_3$ or $w_ku_4$ is in $E$. If $w_ku_3 \not\in E$, then $w_ku_4 \in E$ and so $\langle\{u_4,u_1,u_3, w_k, u_2, b\} \rangle$ is a tailed twin $C_4$ which is forbidden. If $w_ku_3 \in E$ and $w_ku_4 \not\in E$, then $\langle \{ u_1,w_k,u_2,u_3,u_4, a \} \rangle$ is a tailed twin $C_4$ which is forbidden. Hence $w_ku_3, w_ku_4 \in E$.

Thus neither $u_3$ nor $u_4$ belongs to $P[u_2,b]$.

Suppose first that $P[u_2, b]$ has length $2$. Let $v_1$ be its interior vertex. By Theorem \ref{DNBlocal}, $N[v_1]$ is $m^3$-convex. Since $N[v_1]$ contains both $u_2$ and $b$, $v_1$ is adjacent with every vertex on every induced $u_2-b$ path of length at least $3$. So $v_1$ is adjacent with $u_3$ and $u_4$. But now $\langle\{ w_k, u_2, v_1, u_4, b\} \rangle$ is a house which is forbidden.

Suppose now that $P[u_2, b]$ has length at least $3$, say $P[u_2, b]: u_2v_1v_2 \ldots v_rb$. By Theorem \ref{DNBlocal}, $N[\{u_3,u_4\}]$ is $m^3$-convex. Since $u_2, b \in N[\{u_3,u_4\}]$, every vertex of $P[u_2, b]$ is adjacent with either $u_3$ or $u_4$. Let $b=v_{r+1}$. Let $i$ be the smallest integer such that $v_iu_4 \in E$, possibly $i=r+1$. Then $w_ku_2v_1 \ldots v_i u_4 w_k$ is an induced cycle. Since $G$ has no holes $i=1$. Let $j$ be the smallest integer greater than $1$ such that $v_ju_4 \in E$; possibly $j =r+1$. If $j=2$, then $\langle\{ w_k, u_2, v_1, v_2, u_4 \} \rangle$ is a house which is forbidden. Thus $j=3$; otherwise, $u_4 v_1 v_2 \ldots v_ju_4$ is an induced cycle of length at least $5$; which is forbidden. But then $\langle \{ w_k,u_2,v_1,v_2,v_3,u_4 \} \rangle$ is a domino which is again forbidden.

\noindent{\bf Case 2} $u_1 \in V(P[a, u_2])$. By considering $P[u_2, b]$ one can argue as in the previous case that $G$ contains a forbidden subgraph. Hence the lemma follows.
\end{proof}

We now complete the proof of the theorem.
By the above $G$ is (house, hole, domino, $\cal{R}$$_{C_4}$, $\cal{T}$$_{C_4}$)-free.  Suppose $G$ contains the $A$ graph as an induced subgraph. Then the collection of extreme vertices for the convex hull, $CH(A)$, of the $A$ graph is a subset of the set of two leaves of the $A$ graph. By Lemma \ref{monophonic interval of leaves of A graph} the monophonic interval of the leaves of the $A$ graph does not include all the vertices of the $A$-graph. By Lemmas \ref{m-interval is m^3convex} and \ref{m-interval is m_3convex}, $I_m[a,b]$ is $m^3_3$-convex for all $a, b \in V$. This is true in particular for the two leaves of the $A$ graph. Hence the convex hull of the extreme vertices of $CH(A)$ is thus not equal to $CH(A)$. This contradicts the fact that $(V, \cal{M}$$_{m^3_3}(G))$ is a convex geometry.
\end{proof}

\end{document}